\documentclass[12pt,oneside,reqno]{amsart}
\usepackage[utf8]{inputenc}
\usepackage[T1]{fontenc}
\usepackage[margin=.95in]{geometry}
\usepackage{
amsmath,
amssymb,
amsthm,
amsrefs,
mathrsfs,
dsfont, 
xcolor,
graphicx,
float,
enumitem, 
comment,
caption, 
soul 
}
\definecolor{mygray}{gray}{0.3} 
\usepackage[colorlinks=true,urlcolor=gray,linkcolor=mygray,citecolor=mygray]{hyperref}
\newtheorem{theorem}{Theorem}
\newtheorem{proposition}[theorem]{Proposition}
\newtheorem{lemma}[theorem]{Lemma}
\newtheorem{corollary}[theorem]{Corollary}

\newtheorem{definition}{Definition}

\newcommand{\al}{\alpha}
\newcommand{\De}{\Delta}

\newcommand{\eps}{\epsilon}

\newcommand{\la}{\lambda}

\renewcommand{\phi}{\varphi}

\newcommand{\R}{\mathds{R}}
\newcommand{\C}{\mathds{C}}

\newcommand{\PP}{\mathbb{P}}

\newcommand{\E}{\mathbb{E}}

\newcommand{\Sp}{\mathrm{Sp}}

\newcommand{\W}{\mathscr{W}}

\newcommand{\Ov}{\mathscr{O}}

\newcommand{\be}{\begin{equation}}
\newcommand{\ee}{\end{equation}}
\newcommand{\m}{\mathfrak{m}}

\title{Dynamics of a rank-one perturbation\\ of a Hermitian matrix}

\author{Guillaume Dubach}
\address{DMA, École Normale Supérieure -- PSL, 45 rue d'Ulm, F-75230 Cedex 5 Paris, France}
\email{guillaume.dubach@ens.fr}
\author{László Erd\H{o}s}
\address{Institute of Science and Technology Austria, 3400 Klosterneuburg, Austria}
\email{laszlo.erdoes@ist.ac.at}


\begin{document}

\begin{abstract}
We study the eigenvalue trajectories of a time dependent matrix $ G_t = H+i t vv^*$ for $t \geq 0$, where $H$ is an $N \times N$ Hermitian random matrix and $v$ is a unit vector. In particular, we establish that with high probability, an outlier can be distinguished at all times $t>1+N^{-1/3+\eps}$, for any $\eps>0$. The study of this natural process combines elements of Hermitian and non-Hermitian analysis, and illustrates some aspects of the intrinsic instability of (even weakly) non-Hermitian matrices.
\end{abstract}

\maketitle

\begin{figure}[h!]
\includegraphics[width=\textwidth]{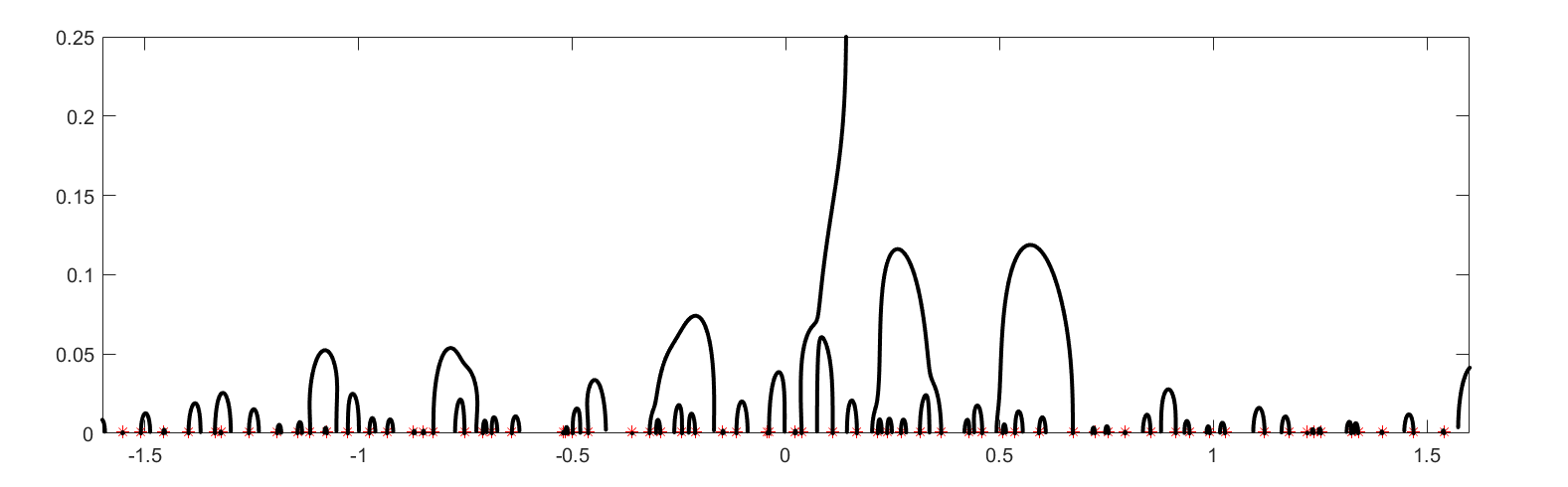}
\caption{Trajectories of bulk eigenvalues for $H$ a $100 \times 100$ GUE matrix.}
\end{figure}

\section*{Introduction}
Rank-one perturbations of random matrices appear naturally in a variety of contexts, an overview of which was recently provided by \cite{Forrester_Review}. An important example is the celebrated phenomenon of \textit{BBP transition} (for Baik-Ben Arous-Péché: see \cites{BBP_paper, Peche_BBP_GUE}), which arises when the perturbation is a positive rank-one Hermitian matrix, and so slightly `pushes' the spectrum to the right, to the point where one outlier is clearly separated from the bulk. Another example is that of non-Hermitian perturbations of a Hermitian matrix, which play an important role in scattering theory (for a general presentation of this application, see Chapter 34 of \cite{RMT_handbook} and \cite{FyodorovSommers1997}). The questions we consider here can be summarized as follows: what can be rigorously established about rank-one anti-Hermitian perturbations of a random Hermitian matrix considered dynamically, that is, when the coupling parameter is interpreted as time? And more specifically, what can be said about the emergence of an outlier? Questions of a dynamical nature have been considered early on in the physics litterature (e.g. \cite{Smolyarenko2003}), and the distribution of eigenvalues (and approximate location of the outlier) in such models has recently been the subject of much mathematical work (see for instance \cites{ORourkeWood2017, Rochet2017, Shcherbina2021}). However, the question of the exact timescale at which the outlier appears seems to have not been adressed until now. \medskip

Throughout the paper, the essential assumption is that $H$ be a random Hermitian matrix for which the uniform isotropic local law (Theorem \ref{thm:UILL}) is known to hold. For the sake of definiteness, say that we consider the Wigner ensemble with the following standard assumptions: entries $h_{ij}, i \leq j$ are independent, off-diagonal (resp. diagonal) entries are identically distributed with continuous distribution on $\C$ (resp. $\R$) such that $\E h_{ij}=0$, $\E |h_{ij}|^2=1/N$ and finite moments, i.e. $\E |\sqrt{N} h_{ij}|^p\le C_p$ for all $p$. These are the Wigner matrices with which we work by default -- although the method and results also hold under more general conditions, such as those of \cite{ErdoesKruegerSchroeder}. As the entries of $H$ are assumed to have a continuous distribution, it holds almost surely that $H$ has $N$ distinct real eigenvalues $\mu_1 < \dots < \mu_N$; and we denote by $(u_i)_{i=1}^N$ a choice of associated  normalized  eigenvectors. \medskip

We consider the following process, which is a rank-one perturbation of $H$:
\begin{equation}\label{def_system}
G_t := H + i t vv^*, \qquad t \in \R,
\end{equation}
where $v$ is a random unit vector, chosen uniformly on the sphere and independent of $H$. However, the randomness of $v$ is not a very relevant feature (as long as independence holds). The main results, indeed, are proved for any fixed $v$; only the proofs of some preliminary facts are greatly simplified when stating them with respect to the randomness in $v$. \medskip

It is straightforward to check that the eigenvalues of $G_t$ lie in the upper half-plane for $t>0$, and that $G_{-t}=G_t^*$, so that the eigenvalue trajectories for $t<0$ and $t>0$ are symmetric to the real axis. 
Another deterministic property is that, as $t \rightarrow \infty$, the spectrum is composed of one outlier that diverges ($\la_{j_{\text{out}}}(t) \approx it$) and $N-1$ eigenvalues that converge to specific locations on the real line.
Note that the distribution of eigenvalues at any fixed $t$ is known when $H$ is taken from an integrable ensemble such as GUE or GOE (see for instance \cites{FyodorovKhoruzhenko,FyodorovSommers}). However, the questions we ask here are of a dynamic nature: we are interested in the evolution of the spectrum $\{ \la_1(t), \dots, \la_N(t) \}$ of $G_t$ when $t \in [0, +\infty) $, and in particular in the emergence of a single outlier (Theorem \ref{thm:the_outlier}). The fact that indices can be given consistently to form $N$ continuous trajectories is a consequence of the non-intersection of trajectories (Theorem \ref{thm:non_crossing}); we choose these indices so that $\lambda_i(0)=\mu_i$. 
\medskip

Section \ref{first_prop_sec} introduces some basic properties of these dynamics, either deterministic or probabilistic. In particular, it is established that trajectories are almost surely non-crossing (Theorem \ref{thm:non_crossing}); moreover, they are everywhere differentiable and satisfy remarkable systems of differential equations of first and second order with singularities (Theorem \ref{thm:ode}). Seen from this angle, the system appears to be extremely unstable, so that an alternative approach is needed. \medskip

Section \ref{subtle_prop_sec} relies on the isotropic local law (Theorem \ref{thm:UILL}), borrowed from the existing literature on Hermitian random matrices, to give more precise high-probability estimates on these trajectories. Most importantly, we establish in Theorem \ref{thm:the_outlier} that the outlier is distinctly separated from the rest of the spectrum at all times $t > 1+N^{-1/3+\eps}$, with $\eps>0$. This timescale for the emergence of the outlier happens to coincide with the critical timescale of BBP transition. Inspired by the present work, Fyodorov, Khoruzhenko and Poplavskyi \cite{FyodorovGUE} provided clear evidence that this timescale is indeed optimal, when $H$ is GUE distributed, based upon an explicit formula for the density of the eigenvalues \cite{FyodorovSommers1996}. \medskip

A natural question that is left open is that of the origin of the outlier: from which eigenvalue $\mu_j$ of $H$ is this particular trajectory more likely to originate? In the context of a Hermitian perturbation, the answer is  trivial;
 for an anti-Hermitian perturbation it becomes very subtle. Heuristic arguments as well as numerical simulations seem to imply that the eigenvalues closer to the origin are much more likely to become the outlier when $t$ increases. However, the absence of a local law very near the spectrum 
prevents us from turning this phenomenology into a rigorous statement.

\subsection*{Notations and conventions}

We introduce the following standard definition.

\begin{definition}[High Probability] 
A sequence of events $(A_N)_{N \geq  1}$ is said to happen with high probability if for any $D>0$ the inequality
$$ 
\PP \left( A_N^c \right) < N^{-D}
$$
holds for sufficiently large $N$.
\end{definition}


It is customary, when working with Wigner  matrices, to define the function $m_{\text{sc}}$ on $\C \backslash [-2,2]$,
the Stieltjes transform of the Wigner semicircle distribution $\rho_{\text{sc}}(x) =\frac{1}{2\pi}\sqrt{4-x^2}$ on $[-2,2]$.
This function is the natural approximation of the resolvent on both the upper and lower half-planes.
Note that $m_{sc}$ has a jump discontinuity on $[-2,2]$ 
and $m_{\text{sc}}(\overline{z}) = \overline{m_{\text{sc}}(z)}$. 
In this paper we  need  a slight modification of this function on the lower half-plane that is holomorphic through $[-2,2]$,
i.e. we define
\be\label{Stieltjes_msc}
\mathfrak{m} (z) =\frac{-z + \sqrt{z^2-4}}{2}
\ee
which is holomorphic on $\C \backslash (-\infty,-2]\cup[2,+\infty)$ with the appropriate choice of branch-cut for the square root, such that $\Im ( \sqrt{z^2-4} ) >0$ for every $z$ in this domain. In particular, $\mathfrak{m}=m_{\text{sc}}$ on the upper-half plane. It is 
a solution to the equation $\mathfrak{m}(z)^2 + z \mathfrak{m}(z)+1 =0$, so that $\mathfrak{m}$ defines a bijection from $\C \backslash (-\infty,-2]\cup[2,+\infty)$ to its image with
\be\label{inverse_m}
z= -\frac{1+\mathfrak{m}(z)^2}{\mathfrak{m}(z)}.
\ee
We also define, for any $t>0$,
\be
t^* := t - \frac{1}{t}, 
\quad \text{such that} \quad 
\mathfrak{m}(it^*) = \frac{i}{t}.
\ee
It is a very important fact (for most results in Section \ref{subtle_prop_sec}) that the holomorphic function $\m(z)-i/t$ has only one zero at $z=it^*$, with multiplicity one; this zero being in the upper-half plane if and only if $t\geq 1$. The role of the quantity $t^*$ as an approximation to the resonant eigenvalue  was already noted, see for instance \cite[p.1950]{FyodorovSommers1997}.

\section*{Acknowledgments}
We would like to thank Paul Bourgade, Victor Dubach, Yan Fyodorov, and Boris Khoruzhenko for many useful remarks. \medskip

\noindent G. Dubach gratefully acknowledges funding from the European Union's Horizon 2020 research and innovation programme under the Marie Sk{\l}odowska-Curie Grant Agreement No. 754411. L. Erd\H{o}s is supported by ERC Advanced Grant “RMTBeyond” No. 101020331.

\newpage
\section{First properties of trajectories}\label{first_prop_sec}

In this section, we consider a fixed matrix $H$. The results are either deterministic, or stated with respect to $\PP_v$, the randomness in $v$, a uniform unit vector.

\subsection{Weighted resolvent and non-intersection of trajectories}

For now, one can assume that the indices of the eigenvalues $\la_j(t)$ of $G_t$ are given arbitrarily for each $t$. One of the goals of the following results is to establish that the indices can be given in a consistant way, with $\la_j(t)$ being `the' trajectory such that $\la_j(0)=\mu_j$.

\begin{definition}
We define the weighted resolvent associated to $H$ and the unit vector $v$ by
\begin{equation}\label{weighted_resolvent}
    \W (z) := \sum_{j=1}^N \frac{|\langle u_j | v \rangle|^2}{\mu_j - z} = \langle v | (H-z)^{-1} v \rangle.
\end{equation}
\end{definition}
The name \textit{weighted resolvent} refers to the fact that  $\W(z)$ can be considered as
a weighted sum with weights $|\langle u_j | v \rangle|^2$ that sum up to one. Since $\E_v |\langle u_j | v \rangle|^2 = \frac{1}{N}$,
where expectation is with respect to the uniform unit vector $v$, 
so $ \E_v \W(z) $ is the usual trace of the resolvent of $H$.


\begin{proposition}\label{level_lines}
For any $t \neq 0$,
\be
z \in \Sp \left( G_t \right)
\quad
\Leftrightarrow \quad \W(z)= \frac{ i }{ t }.
\ee
As a consequence, the trajectories of eigenvalues for the system (\ref{def_system}) are given by the zero level lines of the real part of the weighted resolvent $\W$:
\be
\bigcup_{j=1}^N
\{ 
\la_j(t) \ : \ \ t \in \mathbb{R}^*
\}
=
\{ 
z \in \C \backslash \mathrm{Sp} H \ : \ \Re \W(z) = 0
\}.
\ee
In particular,
\be\label{real_part_bound}
\forall j, t, \quad \mu_1 \leq \Re \la_j(t) \leq \mu_N
\ee
with equality happening only for $t=0$, or if $|\langle u_1 | v \rangle|^2=1$ (resp. $| \langle u_N | v \rangle|^2=1$).
\end{proposition}

\begin{proof}
For any $z \notin \R$, the matrix $(H-z)^{-1} vv^*$ has rank $1$, and in particular its trace $\W(z)$ is its only non-zero eigenvalue. We write:
\be
\det (G_t -z )
= 
\det (H-z) \det\left(I_N + it (H-z)^{-1} vv^* \right)
=
\det (H-z) \left(1 + it \W(z) \right).
\ee
The result follows, as $ z \in \Sp (G_t)$ is  equivalent to $1 + it \W(z)=0$. This is equivalent to $z = E+i \eta$ being a solution to the equation
\begin{equation}\label{algebraic_variety}
    \sum_{j=1}^N |\langle u_j | v \rangle|^2 \frac{\mu_j - E}{(\mu_j -E)^2 + \eta^2} = 0.
\end{equation}
and \eqref{real_part_bound} follows by inspection, all terms having the same sign outside the vertical strip $\{\mu_1 \leq \Re z \leq \mu_N\}$.
\end{proof}

We now prove that the trajectories almost surely do not cross. This is in fact true for any $N$; we give below a concise argument that requires $N \geq 5$, which is enough for our purpose, as subsequent results concern large values of $N$.

\begin{theorem}\label{thm:non_crossing} We assume $N \geq 5$, and that $H$ has distinct real eigenvalues $\mu_1 < \dots < \mu_N$. Almost surely with respect to $\PP_v$, the randomness in $v$, the trajectories of the system \eqref{def_system} do not intersect, nor do they self-intersect: that is, for all $i \neq j$,
\be 
\PP_v \left( 
\exists t,s \geq 0, \ \la_i (t) = \la_j (s)
\right)
= 0 
\quad \& \quad
\PP_v \left( 
\exists t,s \geq 0, t \neq s, \ \la_j (t) = \la_j (s)
\right)
= 0 .
\ee
\end{theorem}

\begin{proof}
If for some $j_1, j_2$, $\la_{j_1} (t) = \la_{j_2} (s) = z$, then by Proposition \ref{level_lines},
\be\label{equality}
\W(z) = \frac{ i }{ t } = \frac{ i }{ s }
\ee
and so $t=s$ (if $t$ or $s=0$, \eqref{equality} is understood as meaning that $z$ is a pole of $\W$). This rules out self-intersection, as well as intersection of two distinct trajectories for different times. Intersection of two trajectories at $t=0$ is ruled out by the eigenvalues of $H$ being distinct. The remaining possibility is that two trajectories intersect at the same time $t>0$, which implies that $\W - i/t$ vanishes at $z \notin \R$ with multiplicity at least $2$; in particular $\W'(z)=0$. This is to say that an intersection point $z$ is such that the two conditions
\begin{equation}\label{two_conditions}
\Re \W (z) = 0
\quad \& \quad
\W ' (z) = 0
\end{equation}
are met. We will prove that almost surely, there is no $z$ that checks both conditions. First, note that
\be
\W'(z) = \sum_{j=1}^{N} \frac{ |\langle u_j | v \rangle |^2 }{ (\mu_j - z)^2 }
\ee
and so the condition $\W'(z) = 0$ is equivalent to $z$ being the root of a real polynomial of degree $2N-2$ (almost surely) that does not vanish on the real line. There is a finite number of such points: namely, almost surely $N-1$ conjugated pairs that we denote $(Z_k,\overline{Z_k})_{k =1}^{N-1}$, with $\Im Z_k >0$, counted with multiplicity.
For any $z \notin \R$, let us define the real vectors:
\begin{equation}\label{three_real_vectors}
Y_1(z) := \left( \Re \frac{1}{ \mu_j -z } \right)_{j=1}^N, \quad
Y_2 (z) := \left( \Re \frac{1}{ (\mu_j -z)^2 } \right)_{j=1}^N, \quad
Y_3 (z) := \left( \Im \frac{1}{ (\mu_j -z)^2 } \right)_{j=1}^N,
\end{equation}
and notice that the first condition of (\ref{two_conditions}) can be written as
\begin{equation}\label{first_condition}
\langle X \ | \ Y_1(z) \rangle = 0     
\end{equation}
where $X := (|\langle u_j | v \rangle |^2)_{j=1}^N$, and the second one similarly as
\begin{equation}\label{second_condition}
\langle X \ | \ Y_2(z) \rangle =0, \qquad  \langle X \ | \ Y_3(z) \rangle =0. 
\end{equation}
We will rely on the (deterministic) fact that $Y_1(z)$ is not in the $\R$-span of $Y_2(z), Y_3(z)$ for any $z \in \C \backslash \R$.

\begin{lemma}
Assuming $N \geq 5$, for any $z = E+i \eta \notin \R$,
$
Y_1(z) \notin \mathrm{Span}_{\R} (Y_2(z), Y_3(z)).
$
\end{lemma}

\begin{proof}
For any $z = E+i \eta$ with $\eta \neq 0$, we denote
\be
a_j + i b_j = \frac{1}{\mu_j - z}.
\ee
In particular, every pair $(a_j, b_j)$ solves
\begin{equation}\label{circle_eta}
b_j = \eta (a_j^2 + b_j^2)
\end{equation}
which is the equation of the circle
$\mathscr{C}_{\eta}$ with center $\frac{i}{2\eta}$ and radius $\frac{1}{2\eta}$, in the $(a,b)$ plane. Moreover, the $\mu_j$'s being distinct, $(a_j,b_j)$ are $N$ distinct points on $\mathscr{C}_\eta$. \\ \\
Moreover, assuming that $Y_1(z) = \al Y_2(z)+\beta Y_3(z)$ for some $\al, \beta \in \R$, we have
\begin{equation}\label{hyperbola}
a_j = \al (a_j^2 - b_j^2) + 2 \beta a_j b_j ,
\end{equation}
which is the equation of a hyperbola $\mathscr{H}_{\al, \beta}$ (including the degenerate case, that yields a union of two lines). By general theory (e.g. Bézout's theorem for curves), we have
\begin{equation}\label{Bezout_curve}
N \leq  | \mathscr{C}_{\eta} \cap \mathscr{H}_{\al,\beta} | \leq 4,
\end{equation}
which is a contradiction.
\end{proof}

For any fixed $z \in \C \backslash \R$, conditionally on (\ref{second_condition}), the probability that (\ref{first_condition}) holds as well is zero, as $Y_1(z)$ is linearly independent of $Y_2(z), Y_3(z)$ and $X$ has a continuous distribution on the simplex 
$$ \left\{ (x_1, \dots, x_N) \ : \ \ x_j \geq 0 , \ \sum_{j=1}^N x_j =1 \  \right\}.$$
So what the above argument allows to conclude is that
\be
\forall z \in \C \backslash \R, \
\PP_v \left( \Re \W(z) = 0 \ | \ \W'(z)=0 \right) = 0.
\ee
Considering that there are finally many points $(Z_i, \overline{Z}_i)_{i=1}^{N-1}$ such that $\W'(z)=0$, and that these can be assumed to be exchangeable (for instance, by reshuffling their indices by a uniform random permutation), it follows by classical probabilistic arguments that
\be\label{no_z_exists}
\PP_v \left( \exists z \in \C \backslash \R, \ 
\Re \W (z) = 0
\ \& \
\W ' (z) = 0  \right) = 0.
\ee
which concludes the proof.
\end{proof}

An important consequence of Theorem \ref{thm:non_crossing} is the possibility of choosing an coherent indexation such that each $\la_j(t)$ is a (uniquely defined) distinct continuous trajectory. We now study the deterministic behavior of these trajectories.

\subsection{Deterministic evolution}

The main features of these $N$ almost surely non-crossing continuous trajectories are as follows:
\begin{enumerate}
\item[(i)] At $t=0$, all eigenvalues are real.
\item[(ii)] For $t>0$, all eigenvalues are in the upper-half plane.
\item[(iii)] When $t \rightarrow \infty$, one eigenvalue (`the outlier' $\la_{j_{\text{out}}}$) diverges with $\Im \la_{j_{\text{out}}} \xrightarrow[t \rightarrow \infty]{} \infty$ and bounded real part, and the rest of the spectrum converges to $N-1$ distinct points on the real line. 
\end{enumerate}
The first two properties immediately follow from the definition of $G_t$. The last one can be easily established by the Schur complement identity, that also allows to identify the limit points as the eigenvalues of the projection of the operator $H$ on the space orthogonal to $v$. Another remarkable deterministic fact is that this evolution of eigenvalues can be described by two closed systems of differential equations: indeed, both first and second derivatives can be expressed in terms of lower order terms, as we now state.

\begin{theorem}[First and Second Order Differential Equations]\label{thm:ode} Let $H$ be a Hermitian matrix with simple eigenvalues $(\mu_j)_{j=1}^N$ and associated unit eigenvectors $(u_j)_{j=1}^N$. The evolution of the eigenvalues $(\la_j(t))_{j=1}^N$ of $G_t = H+i t vv^*$ can be described by the following closed system of equations, as long as the eigenvalues are distinct\footnote{This is almost surely the case when $v$ is a random unit vector, by Theorem \ref{thm:non_crossing}.}. The initial condition is $\la_j(0) = \mu_j $. For $t=0$, one has
\begin{equation}\label{initial_push}
\la_j '(0) = i |v^* u_j |^2 ,
\end{equation}
and for $t>0$, 
\begin{equation}\label{evolution_eq}
\la_j '(t) 
= \frac{i \Im \la_j(t)}{t} \prod_{k \neq j} \frac{\la_j(t) - \overline{\la_k}(t)}{\la_j(t) - \la_k(t)}.
\end{equation}
Moreover, the following second order equation holds, for $t>0$:
\begin{equation}\label{second_order}
\la''_j (t) = 2 \la'_j (t) \sum_{k \neq j} \frac{\la'_k(t)}{\la_j - \la_k}.
\end{equation}
\end{theorem}
Note that the product in the right hand side of \eqref{evolution_eq} is exactly the value of the diagonal overlaps $\Ov_{jj}$ (see formula (11) in \cite{FyodorovMehlig}). Remarkably, the second order equation \eqref{second_order} is valid more generally for the eigenvalues of $G_t^{(\theta)} = H+ e^{i \theta} t vv^*$ with any $\theta$ including the fully Hermitian case $\theta=0$; the proof is the same.


\begin{proof}
If the row vectors $L_i$ and the column vectors $R_j$ are respectively the left and right eigenvectors of $G_t$, chosen with the biorthogonality condition
\begin{equation}\label{bi_orthog}
L_i R_j = \delta_{ij},
\end{equation}
which is always possible when the corresponding eigenvalues are distinct, and if we call $X$ the matrix with columns $R_1, \dots, R_N$ and $Y$ the matrix with rows $L_1, \dots, L_N$, then it follows in particular that:
\begin{equation}
Y X =I, \quad G_t X = X \Delta, \quad Y G_t = \Delta Y .
\end{equation}
where $\Delta = \rm{Diag}(\lambda_1, \dots, \lambda_N)$. Differentiating with respect to $t$ yields
\begin{equation}\label{diff_1}
\Delta' = Y G_t' X + [\De, Y X'].
\end{equation}
In the present case, $G'_t = i v v^*$, and so equation (\ref{diff_1}) gives, on the diagonal,
\begin{equation}\label{first_derivative1}
\la_j ' (t)  =  L_j  G'_t  R_j 
= i  L_j  v  v^* R_j .
\end{equation}
For $t=0$, $G_0=H$ and $R_j=L_j=u_j$, so that equation (\ref{initial_push}) follows. For $t>0$, we first notice that the quantity $ L_1  v v^* R_1 $ is invariant under a unitary change of basis, so we can compute it for a Schur form of $G_t$. As this Schur form $T$ is conjugated to $G_t$ by a unitary change of variable, we have 
\be 
T = \widetilde{H} + it \tilde{v} \tilde{v}^*
\ee
with $\widetilde{H} = U H U^*$, $\tilde{v} = Uv$. We will simply continue to denote these by $H$ and $v$ in order to not overload notations. 
As $T$ is upper-triangular, we have
\be 
\forall a<b, \quad T_{ba} =  H_{ba} + i t v_b \overline{v_a} =0,
\ee
which implies
\be\label{columns_T}
\forall a<b, \quad  T_{ab} 
= H_{ab} + it v_a \overline{v_b}
= \overline{H}_{ba} + it v_a \overline{v_b} 
= 2it v_a \overline{v_b},
\ee
and on the diagonal, $T_{aa} = \la_a = H_{aa} + it |v_a|^2$, which implies
\be\label{real_im_part}
\Re \la_a = H_{aa}, \qquad \Im \la_a = t |v_a|^2.
\ee
This Schur form can be chosen so that any given eigenvalue is the first on the diagonal, and so we work now with $\la_1$ 
(i.e. we prove \eqref{evolution_eq} for $j=1$) without loss of generality. Another consequence of $T$ being triangular is that $R_1=e_1$ and so $ v^* R_1 = \overline{v_1}$. We now compute $ L_1 v $; in the following argument, we denote $L_1 = (1,\ell_2, \dots,\ell_N)$, and $L_1^{(d)}=(1,\ell_2, \dots, \ell_d)$ so that $L_1=L_1^{(N)}$; similarly $v^{(d)}$ stands for $(v_1, \dots, v_d)$. The numbers $\ell_k$ satisfy a simple recursion, which follows from the definition of $L_1$, and $T$ being triangular. Together with \eqref{columns_T}, this gives
\be\label{Schur_rec}
\ell_{k+1} = \frac{1}{\la_1 - \la_{k+1}}  L_1^{(k)} \tau_{k+1}
 = \frac{2 i t \overline{v_{k+1}}}{\la_1 - \la_{k+1}}  L_1^{(k)} v^{(k)},
\ee
where $\tau_{k+1}$ is the column vector of the first $k$ entries of the $(k+1)$th column of $T$. The recursion for $ L_1 v $ is then initiated by
$$
 L_1^{(1)} v^{(1)} = v_1,
$$
and continued in the following way:
\begin{align*}
 L_1^{(k+1)} v^{(k+1)} & =  L_1^{(k)} v^{(k)} +\ell_{k+1} v_{k+1} \\
& = L_1^{(k)} v^{(k)} + \frac{2 i t |v_{k+1}|^2}{\la_1 - \la_{k+1}} L_1^{(k)}v^{(k)} \\
& = L_1^{(k)} v^{(k)} \left( 1 + \frac{2 i t |v_{k+1}|^2}{\la_1 - \la_{k+1}} \right)
\end{align*}
where we replaced $\ell_{k+1}$ using (\ref{Schur_rec}). 
Eq. (\ref{real_im_part}) gives us $2 i t |v_{k+1}|^2 = \la_{k+1} - \overline{\la_{k+1}}$, so that
\begin{equation}
    L_1^{(k+1)} v^{(k+1)}
    = L_1^{(k)} v^{(k)} \left( 1 + \frac{\la_{k+1} - \overline{\la_{k+1}}}{\la_1 - \la_{k+1}} \right)
    = L_1^{(k)} v^{(k)}  \frac{\la_{1} - \overline{\la_{k+1}}}{\la_1 - \la_{k+1}},
\end{equation}
and finally
$$
 L_1 v  =  L_1^{(N)}  v = v_1 \prod_{k=2}^N \frac{\la_{1} - \overline{\la_{k}}}{\la_1 - \la_{k}}.
$$
It now follows from~\eqref{first_derivative1} that
$$
\la_1 ' (t) = i L_1  v v^* R_1  = i |v_1|^2  \prod_{k=2}^N \frac{\la_{1} - \overline{\la_{k}}}{\la_1 - \la_{k}},
$$
and eq. (\ref{real_im_part}) allows us to obtain the equation (\ref{evolution_eq}), that is a function of eigenvalues only, valid for any $t>0$. 

In order to prove of \eqref{second_order}, we look at the off-diagonal terms of \eqref{diff_1}:
\begin{equation}\label{diff_1_off_diag}
\forall i \neq j, \ 
(\la_i - \la_j) L_i {R}_j' = - L_i G_t' R_j
\end{equation}
So that, expressing the derivative of right (resp. left) eigenvectors in the basis of the right (resp. left) eigenvectors,
\begin{equation}\label{eigenvector_derivatives}
{R}_j' = \sum_{k=1}^N \al_{j,k} R_k, \quad \& \quad
{L}_j' = \sum_{k=1}^N \beta_{j,k} L_k
\end{equation}
we find 
\begin{equation}
\forall k \neq j, \quad
\al_{j,k} = L_k {R}_j' =  \frac{1}{\la_j - \la_k} (L_k G_t' R_j)
\end{equation}
and
\begin{equation}
\forall k \neq j, \quad
\beta_{j,k} = {L}_j' R_k = - L_j {R}_k'  = \frac{1}{\la_j - \la_k} (L_j G_t' R_k)
\end{equation}
we also note that as $L_k R_k = 1$, $\beta_{kk} = L_k'R_k = - L_k R_k' = -\al_{kk}$,
and so
$$
{R}_j' = \sum_{k \neq j} \frac{1}{\la_j - \la_k} (L_k G_t' R_j) R_k + \al_{jj} R_j
\quad \& \quad
{L}_j' = \sum_{k \neq j} \frac{1}{\la_j - \la_k} (L_j G_t' R_k) L_k - \al_{jj} L_j
$$
Differentiating eq. (\ref{diff_1}) on the diagonal using that  $G_t' = i v v^*$, $G_t''=0$:
\begin{align*}
\la_j''(t) & = i L_j ' v v^* R_j +  i L_j v v^* R_j' 
=  i \sum_{k \neq j} \left(\beta_{j,k} (L_k v v^* R_j) + \al_{j,k} (L_j v v^* R_k) \right) \\
& = \sum_{k \neq j} \frac{2}{\la_j - \la_k} (i L_j v v^* R_k)(i L_k v v^* R_j)\\
& = \sum_{k \neq j} \frac{2}{\la_j - \la_k} (i  L_j v v^* R_j)(i  L_k v v^* R_k)
= 2 \la_j ' (t) \sum_{k \neq j} \frac{\la_k'(t) }{\la_j - \la_k},
\end{align*}
which is \eqref{second_order}.
\end{proof}

\section{Properties of trajectories via isotropic local law}\label{subtle_prop_sec}

In this section we state a few estimates on the trajectories of the system \eqref{def_system}, based on the approximation of the weighted resolvent $\W(z)$. The unit vector $v$ is assumed to be fixed in this entire section, and all results are stated with respect to $\PP_H$, the randomness of $H$, a Wigner matrix.

We let $T \geq 2$ be an arbitrary fixed constant; the small time ($t \leq T$) and large time ($t>T$) behavior will be analysed somewhat differently.

For any $\zeta, L>0$, we consider the following spectral domains:
\be\label{S_strip}
\mathscr{S}_{\zeta} := \{ z=E+i\eta \in \C \ : \ |E| < 3, \  N^{-1+\zeta} \leq \eta < N^{100} \},
\ee
\be\label{S_L_strip}
\mathscr{S}_{\zeta,L} := \{ z=E+i\eta \in \C \ : \ |E| < 3, \  N^{-1+\zeta} \leq \eta < L \},
\ee
and
\be\label{R_rectangle}
\mathscr{R}_{ \zeta} :=  \{ z=E+i\eta \in \C \ : \ |E| < 3, \ 0 \leq \eta < N^{-1+\zeta}\}.
\ee 

The essential input in this section is the uniform isotropic local law, taken over from the existing literature. Local laws in general aim at approximating the resolvent by some deterministic quantity ($m_{\text{sc}}$ times the identity matrix for Wigner matrices); `isotropic' refers  to scalar products $\langle v, (H-z)^{-1}v\rangle$ with some fixed vector $v$, i.e. to the weighted resolvent $\W$, and `uniform' refers to uniformity in the parameter $z$. An isotropic local law, for a given $z$ in a bounded domain, was first given in \cite{KnowlesYin2013}.  We need its following version: 

\begin{theorem}[Uniform Isotropic Local Law]\label{thm:UILL}
For any $\zeta, \eps, D>0$ and fixed unit vector $v$,
\be\label{UILL_bound}
\PP_H \left(
\exists z \in \mathscr{S}_\zeta, \
|\W(z) - \mathfrak{m}(z)| > \frac{N^{\eps}}{\sqrt{N\eta} (1+\eta^2)^{3/4}} 
\right)
< N^{-D}.
\ee
\end{theorem}

\begin{proof} 
In the bulk spectrum, $|\Re z|\le 2-\epsilon$, this result was stated in Thm 2.1,  eq. (2.6a) and (2.7a) of \cite{ErdoesKruegerSchroeder} even for much more general Hermitian random matrices with possibly correlated entries. The edge regime was settled in Eq. (2.6a) of \cite{AltErdoesKruegerSchroeder}, where the optimal bound is in fact slightly better
than~\eqref{UILL_bound}. 
Together, these references provide an isotropic local law for any fixed $z \in \mathscr{S}_{\zeta}$. Uniformity in $z$ can be achieved by 
Lipschitz continuity of the functions at stake and using a dense grid of fixed spectral parameters. 
\end{proof}

We will use this isotropic local law together with the following classical theorem.

\paragraph{Rouché's Theorem:}
\textit{
Let $f$  and $g$ be two holomorphic functions on a domain $\Omega\subset\C$ with closed and simple boundary $\partial \Omega$. If $|f(z)-g(z)|<|g(z)|$ on $\partial \Omega$, then $f$ and $g$ have the same number of zeros in $\Omega$, counted with multiplicity.
} \medskip

The following proofs all have in common that we determine a domain on which the inequality $|f-g| < |g|$ holds, with $f= \W - i/t$ and $g=\m-i/t$. The conclusion then either follows immediately, as this strict inequality clearly prevents $f=0$ on the relevant domain, or by applying Rouché's theorem on a Jordan curve: the (deterministic) zeros of $g$ being known, this allows us to draw some conclusions as to the (random) zeros of $f$.

\begin{theorem}\label{thm:elliptic_region}
For any $\eps, \zeta >0$, it holds with $\PP_H$-high probability that all eigenvalues $\lambda_i(t)$, for any time $t \in (0,T)$, lie in the domain $\mathscr{E}_{t, \eps} \cup \mathscr{R}_{\zeta}$, where
\be\label{elliptic_curve}
\mathscr{E}_{t, \eps} := \{ z=E+i\eta \in \C \ : \
E^2 + \left(\eta - t^* \right)^2 < \frac{N^{\eps}}{N \eta}, \ |E| < 3, 0 \leq \eta < T \}.
\ee
\end{theorem}

\begin{corollary}\label{thm:hyperbolic_region}
For any $\eps >0$, with $\PP_H$-high probability, all trajectories up to time $T$ are in the domain
\be 
\mathscr{H}_{\eps} := 
\left\{ z = E+i\eta \in \C \ : \ \eta E^2 < N^{-1+\eps}, \ |E| <3, \ 0 \leq \eta < T \right\}
.
\ee
\end{corollary}

\begin{proof}[Proof of Theorem \ref{thm:elliptic_region}]
It was proved in~\eqref{real_part_bound} that $\mu_1 \leq \Re \la_j(t) \leq \mu_N$ for all $j$ and $t>0$, and it holds with $\PP_H$-high probability that $-3< \mu_1 < \mu_N < 3$ (any number larger than $2$ would do), so that the bound $| \Re \la_j(t)| <3$ follows immediately. Similarly, with probability one, $0< \Im \la_j (t) < t < T$ for all $j$. \medskip

We now justify the main inequality, which results from the following observations. \\
First, by the uniform isotropic law (Theorem \ref{thm:UILL}) for any $\eps, \zeta>0$, with $\PP_H$-high probability
\be\label{up}
\forall z \in \mathscr{S}_\zeta, \quad
 |\W(z) - \mathfrak{m}(z)| < \frac{N^{\eps/4}}{\sqrt{N \eta}}.
\ee
Second, using \eqref{inverse_m}, we write:
\begin{equation}\label{lipstick}
|z_1 - z_2|
 = \left|\frac{1+\m(z_1)^2}{\m(z_1)} - \frac{1+\m(z_2)^2}{\m(z_2)} \right|
\leq |\m(z_1) - \m(z_2)| \left(1 + \frac{1}{|\m(z_1) \m(z_2)|} \right) .
\end{equation}
We apply this to $z_2 = it^*$, $z_1=E+i\eta $ with $|E|<3, t<T, \eta < 2T$ and conclude that in the bounded domain $\mathscr{S}_{\zeta, 2T}$,
\be\label{mi}
|z-it^*| < C_T \left| \mathfrak{m}(z) - \frac{i}{t} \right|
\ee
with a constant $C_T$ that only depends on $T$. 
Finally, for any $z \in \mathscr{S}_{\zeta,2T} \backslash \mathscr{E}_{t,\eps}$, we have (by definition of $ \mathscr{E}_{t,\eps}$)
\be\label{last}
|z-it^*| > \frac{N^{\eps/2}}{\sqrt{N \eta}},
\ee
and therefore the following sequence of inequalities holds for $z \in \mathscr{S}_{\zeta,2T} \backslash \mathscr{E}_{t,\eps}$, bringing together \eqref{up}, \eqref{last}, and \eqref{mi}.
\be\label{put_together}
|\W(z) - \mathfrak{m}(z)| 
< \frac{N^{\eps/4}}{\sqrt{N \eta}} 
< N^{-\eps/4} |z-it^*|
< \left| \mathfrak{m}(z) - \frac{i}{t} \right|, 
\ee
using $C_T\le N^{\eps/4}$.
As noticed above (see \eqref{inverse_m}), $\m(z)-i/t$ has only one zero at $z=it^*$, which is trivially outside $ \mathscr{S}_\zeta \backslash \mathscr{E}_{t,\eps}$. The conclusion is that $\W(z)$ does not take the value $i/t$ on $ \mathscr{S}_\zeta \backslash \mathscr{E}_{t,\eps}$, which is to say that all eigenvalues at time $t$ are in $\mathscr{E}_{t,\eps} \cup \mathscr{R}_{\zeta}$.
\end{proof}

\begin{proof}[Proof of Corollary \ref{thm:hyperbolic_region}]
We apply Theorem \ref{thm:elliptic_region} and note that
\be
\bigcup_{0<t<T} \mathscr{E}_{t,\eps} \subset \mathscr{H}_{\eps},
\ee
which proves that for any $\eps, \zeta >0$ all trajectories lie in $\mathscr{H}_{\eps} \cup \mathscr{R}_{\zeta}$; choosing $\zeta < \eps$ ensures that $\mathscr{R}_{\zeta} \subset \mathscr{H}_{\eps}$, so that with $\PP_H$-high probability all trajectories lie in $\mathscr{H}_{\eps}$. 
\end{proof}

\begin{figure}[h!]
\includegraphics[width=\textwidth]{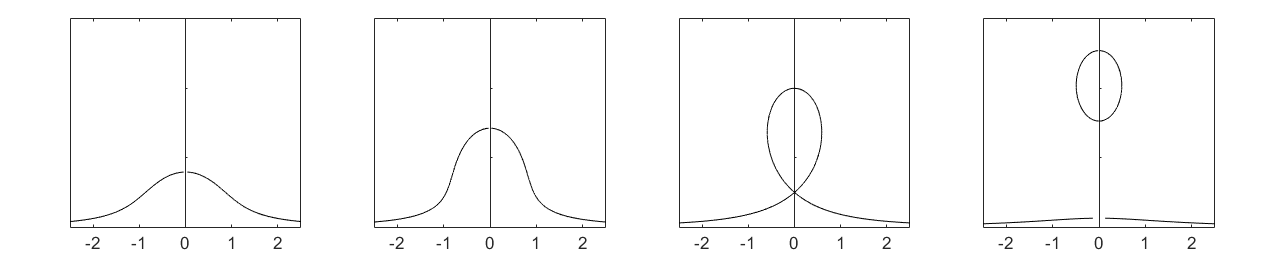}
\caption{Shape of the domain \eqref{elliptic_curve} when $t$ increases (picture not to scale). The first two images correspond to some $t_1 < t_2 <1+N^{-1/3-\eps}$; the last one corresponds to $t_4 > 1+N^{-1/3+\eps}$, when the outlier can be isolated by applying Rouché's theorem around the upper connected component.}
\label{elliptic_fig}
\end{figure}

The above argument relied directly on the strict inequality \eqref{put_together}. For the next result, we rely on Rouché's theorem in order to isolate one particular eigenvalue, which we can do as soon as the relevant domain has a bounded connected component in $\mathscr{S}_{\zeta,2T}$, as illustrated on Fig.~\ref{elliptic_fig}.

\begin{theorem}[Emergence of an outlier]\label{thm:the_outlier}
For any $\eps >0$, with $\PP_H$-high probability, at all times $t \in ( 1+ N^{-1/3+\eps}, T)$ the outlier is in the disk $D \left( it^*,
\frac{N^{\eps/4}}{\sqrt{N t^*}} \right)$, whereas all other eigenvalues  satisfy $\Im \la_j (t)<\frac{N^{\eps}}{N(t^*)^2}$. In particular they are well separated from the outlier\footnote{Recall that $t^*=t-1/t$, and so $t^* \sim 2(t-1)$ in any regime s.t. $t \sim 1$.}.
\end{theorem}

\begin{proof}
We consider some $\eps'>0$ such that $\eps'< \eps/2$, and the domain $\mathscr{E}_{t,\eps'}$ similarly as in \eqref{elliptic_curve}.
We go through the same steps as in the proof of Theorem \ref{thm:elliptic_region} and note that the inequalities \eqref{up}, \eqref{mi}, \eqref{last}, and therefore also \eqref{put_together} hold on $\mathscr{S}_{\zeta,2T} \backslash \mathscr{E}_{t, \eps'}$, allowing us to invoke Rouché on any Jordan curve inside this domain. \\ \\
We further note that, for $t> 1+ N^{-1/3+\eps}$ and $N$ large enough, the domain $\mathscr{E}_{t, \eps'}$ has two connected components, which is a direct calculation. Let us prove that one connected component lies in the disk $D \left( it^*,
\frac{N^{\eps/2}}{\sqrt{N t^*}} \right)$. First, $it^* \in \mathscr{E}_{t, \eps'}$ by inspection. Then, note that $t>1+N^{-1/3+\eps}$ implies $t^* > N^{-1/3+\eps}$, so that for $N$ large enough,
\be
\frac{N^{\eps/4}}{\sqrt{Nt^*}} < \frac13 {t^*},
\ee
which implies that for any $z=E+i\eta \in \partial D(it^*, \frac{N^{\eps/4}}{\sqrt{N t^*}})$,
\be
\eta \geq t^* - \frac{N^{\eps/4}}{\sqrt{Nt^*}} > \frac23 t^*
\ee
and so, for any point on that disk,
\be
E^2 + (\eta-t^*)^2 = \frac{N^{\eps/2}}{Nt^*} > \frac{N^{\eps'}}{N \eta}.
\ee
This and other direct considerations show that $\partial D(it^*, \frac{N^{\eps/4}}{\sqrt{N t^*}}) \subset \mathscr{S}_{\zeta,2T} \backslash \mathscr{E}_{t, \eps'}$. We can apply Rouché on this circle. Owing to the fact that $\m - i/t$ has only one zero at $it^*$, this proves the first statement about the outlier. \\ \\
The second statement follows from checking that the second connected component is below height $\frac{N^{\eps}}{N (t^*)^2}$. First, as $t^*> N^{-1/3+\eps}$, we have
\be
t^* - \frac{N^{\eps}}{N (t^*)^2} > \frac23 t^*.
\ee
Therefore, for any $z=E+i \frac{N^{\eps}}{N (t^*)^2}$,
\be
\eta E^2 + \eta (\eta -t^*)^2 > \frac49 (t^*)^2 \frac{N^{\eps}}{N (t^*)^2} > \frac{N^{\eps'}}{N}
\ee
for any $\eps'<\eps$. This is valid outside $\mathscr{R}_{\zeta}$ for some $\zeta>0$; the second statement follows by choosing $\zeta< \eps$.
\end{proof}

We finally present two complements of our main theorem, before and after the timescale at which the outlier can be isolated. The bounds we obtain for small $t$, up to slightly below the relevant $1+N^{-1/3}$ timescale, are given in Proposition \ref{prop:small_t}, whereas the bounds for $t \geq T$ 
are given in Proposition \ref{prop:large_t}. 

\begin{proposition}[Small $t$ bounds]\label{prop:small_t}
For any $\zeta,\eps>0$, with $\PP_H$-high probability, at all times $t < 1+N^{-1/3-\eps} $, all eigenvalues satisfy $\Im \la_j(t) < N^{-1/3 + \eps}$. Moreover, for $t < 1- N^{-1/3 + \eps}$,
\be\label{bound2}
\forall j \quad \Im \la_j(t) < \max \left(\frac{N^{\eps}}{N (t^*)^2},\frac{N^{\zeta}}{N} \right).
\ee
\end{proposition}

\begin{proof}
These bounds are direct consequences of Theorem \ref{thm:elliptic_region} applied for a well chosen $\eps'$. For instance, if $t < 1$, then $t^* <0$ and the inequality in \eqref{elliptic_curve} implies
$ \eta^3 < \eta (\eta-t^*)^2 < N^{-1+\eps'},$
from which the bound $\Im \la_j < N^{-1/3+\eps}$ follows if $\eps'<\eps$. It is a calculus exercise to check that this inequality still holds as long as $1 \leq t < 1+N^{-1/3-\eps}$, when choosing an appropriate $\eps'$. \\

If $t< 1-N^{-1/3+\eps}$, it can be directly checked that the domain \eqref{elliptic_curve} is connected and contains the origin, so that any line that is not contained in it actually bounds it. Together with the fact that in that regime,
$$
\frac{N^{\eps}}{N (t^*)^2} \ll |t^*|
$$
one can check that the line $\{ \eta = \frac{N^{\eps}}{N (t^*)^2} \}$ is not in $\mathscr{E}_{t, \eps}$, and therefore all eigevalues are below this threshold, or in $\mathscr{R}_{\zeta}$.
\end{proof}

\begin{proposition}[Large $t$ bounds]\label{prop:large_t}
For any $\zeta>\eps>0$, with $\PP_H$-high probability, at all times $t \in [T,N^{99}]$, the outlier is in the 
small disk $D(it^*, N^{-1/2+\eps})$, while all other eigenvalues are in $\mathscr{R}_{\zeta}$.
\end{proposition}

\begin{proof}
The only difference with the previous proofs is the change of domain, from a bounded one close to the real line, 
to a domain 
far away from the real line. There are two 
consequences of this change: on the one hand, the isotropic law gives a better bound, but on the other hand the inequality \eqref{mi} 
has to be replaced by a weaker one. For this, we apply the inequality \eqref{lipstick} this to $z_2 = it^*$, $z_1=E+i\eta$ with $\eta \in [t/2,2t]$,
\be 
|z - it^*| < \left|\m(z) - \frac{i}{t}\right| \left(1 + \frac{t}{|\m(z)|} \right)
< N^{\eps/2} \eta^2 \left|\m(z) - \frac{i}{t}\right|.
\ee
In the domain
\be 
\mathscr{F}_{t, \eps} := \{ |E|<3, \ t/2 < \eta < 2t, \ |z-it^*| \geq  N^{-1/2+\eps} \},
\ee
which is a rectangle with a small disk removed, we have the sequence of inequalities:
\be 
|\W(z) - \mathfrak{m}(z)| 
< \frac{N^{\eps/2}}{\eta^2 \sqrt{N}} 
\leq N^{-\eps/2} \frac{|z-it^*|}{\eta^2}
< \left| \mathfrak{m}(z) - \frac{i}{t} \right|
\ee
and so we can apply Rouché on $\partial D(it^*, N^{-1/2+\eps})$. We recall that the comparison function $\m - i/t$ has exactly one root, which is at $it^*$. This proves that $\W-i/t$ also has exactly one root inside the disk, which is the outlier. As for the other eigenvalues, the argument from Theorem \ref{thm:the_outlier} works the same; as we assume $\zeta> \eps$, the region $\mathscr{R}_{\zeta}$ contains the second connected component of \eqref{elliptic_curve}.
\end{proof}

\end{document}